\documentclass{amsart} 

\usepackage{xcolor}

\usepackage{lmodern}
\usepackage{babel}
\usepackage{amssymb, latexsym, amsmath, amsfonts, hyperref, enumerate,  fancyhdr}
\usepackage{amsmath,amscd}
\usepackage{cleveref}
\usepackage{mathrsfs}
\usepackage{mathtools}
\usepackage{a4wide}

\usepackage[english=usenglishmax]{hyphsubst}
\usepackage{microtype}

\newcommand{\m}{\mathcal}

\renewcommand{\leq}{\leqslant}
\renewcommand{\geq}{\geqslant}
	\title[]{Failure of almost uniform convergence for noncommutative martingales}
	
	

\author{Guixiang Hong}
\author{\'Eric Ricard}
\address{Institute for Advanced Study in Mathematics, Harbin Institute of Technology, Harbin 150001, China.}
\email{gxhong@hit.edu.cn}
\address{UNICAEN, CNRS, LMNO, 14000 Caen, France}
\email{eric.ricard@unicaen.fr}

\theoremstyle{plain}
\newtheorem{thm}{Theorem}[section]

\newtheorem{lem}[thm]{Lemma}

\newtheorem{prop}[thm]{Proposition}

\theoremstyle{definition}
\newtheorem{defn}[thm]{Definition}

\begin{document}

\begin{abstract}
In this paper, we provide a counterexample to show that in sharp contrast to the classical case, the almost uniform convergence may not happen for truly noncommutative $L_p$-martingales when $1\leq p<2$. The same happens to ergodic averages. The proof consists of some sharp estimates of the distributional function of a sequence of matrices and some non standard transference techniques, which might admit further applications.

\end{abstract}

\maketitle

%\tableofcontents

\section{Introduction}\label{sec:intro}
On the truly noncommutative measure spaces $(\mathcal M, \tau)$, the
investigation of the notion `almost everywhere convergence' dates back
to the birth of noncommutative integration theory
(cf. \cite{Seg52}[Definition 2.3]), see also \cite{Pad67}[Theorem
3.1], \cite{Rad73}. But the first significant result was obtained by
Lance \cite{Lan76}[Theorem 5.7], where he proved that the ergodic
averages $$M_nx:=M_n(T)(x)=\frac1n\sum^{n-1}_{k=0}T^kx$$ with initial
data $x\in \mathcal M=L_\infty(\mathcal M)$ converge {\it almost
  uniformly} (a.u. in short) when $T$ is a trace preserving
$*$-automorphism; more
precisely, there exists one $\hat{x}\in \mathcal M$ such that for any
$\varepsilon>0$ , there exists some projection $e\in \m M$ such that
    \begin{align}\label{au}\tau(1-e)<\varepsilon \qquad \qquad \textrm{and }\quad \quad \lim_{n\to \infty} \| (M_nx-\hat{x}) e\|_\infty =0.\end{align}
    Later on, in 1977 Yeadon \cite{Yea77}[Theorem 2] showed that the
    ergodic averages of a unital trace preserving map with initial
    data $x\in L_1(\mathcal M)$ converge {\it bilaterally almost
      uniformly}, that is, there exists one
    $\hat{x}\in L_1(\mathcal M)$ such that for any $\varepsilon>0$ ,
    there exists some projection $e\in \m M$ such that
    \begin{align}\label{bau}\tau(1-e)<\varepsilon \qquad \qquad \textrm{and }\quad \quad \lim_{n\to \infty} \| e(M_nx-\hat{x}) e\|_\infty =0;\end{align}
    and thus recovered the classical famous individual (or pointwise)
    ergodic theorem of Birkhoff for integrable functions. Yeadon's
    bilateral almost uniform convergence was also partially motivated
    by Cuculescu's fundamental result \cite{Cuc71}[Proposition 6],
    where any noncommutative bounded $L_1$-martingale was showed to
    converge bilaterally almost uniformly. As in the classical
    setting, all these (bilateral) almost uniform convergence results
    were deduced from some noncommutative maximal inequalities. Around
    thirty years later, along with several breakthroughs made on
    noncommutative maximal inequalities \cite{Pis98, Jun02, JuXu07,
      HJP16}, it is now well-known that the almost uniform convergence
    of ergodic averages is still true for initial data
    $x\in L_p(\mathcal M)$ with $p\geq2$ and the same happens for
    bounded $L_p$-martingales with $p\geq2$.  Bilateral almost
    uniform convergence holds when $p\geq 1$.

    Note that the notion `almost uniform convergence' is stronger than
    `bilateral almost uniform convergence', it is then natural to
    pursue the almost uniform convergence of ergodic averages or
    bounded $L_p$-martingales for $p<2$. As far as we know, several
    experts have been trying hard to solve this problem in a
    positive way. But in this paper, we provide a counterexample to
    show that in sharp contrast to the classical case, the almost
    uniform convergence can fail for truly noncommutative
    $L_p$-martingales when $1\leq p<2$.

\begin{thm}\label{thm:nonau}
  Let $1\leq p<2$. There exists a probability space
  $(\mathcal M,\tau)$ and a sequence of von Neumann subalgebras
  $\mathcal M_n$ with associated conditional expectations
  $\mathcal E_n$ and $\mathcal X\in L_p(\mathcal M)$ such that
  $(\mathcal E_n(\mathcal X))_{n\geq 0}$ does not converge almost uniformly.
\end{thm}

Some results about a.u. convergence for martingale Hardy spaces have
been shown in \cite{HJP16} (see the definitions there). When
$1\leq p<2$, if
$\mathcal X\in H_p^c(\mathcal M)\subset L_p(\mathcal M)$ then
$(\mathcal E_n(\m X))_{n\geq 0}$ converges a.u. to $\mathcal X$. Since
$L_p(\mathcal M)=H_p^c(\mathcal M)+H_p^r(\mathcal M)$, $1<p<2$, it
follows that there exists $\mathcal X\in H_p^r(\mathcal M)$ such that
$(\mathcal E_n(\m X))_{n\geq 0}$ does not converges a.u. (but $(\mathcal E_n(\m X^*))_{n\geq 0}$
does).

With more effort, we show that ergodic averages may also fail to converge almost
uniformly when $1\leq p<2$ see Section \ref{sec:erg}. Together with the noncommutative Banach principle (cf. \cite{HLW21}[Section 6]), one concludes also that the strong asymmetric maximal inequalities  cannot be true for ergodic averages when $1\leq p<2$, see e.g. \cite{HJP16} for more information on martingales.

\section{Preliminaries}

%Lance BAMS76, Inventiones
%Cuculescu 71, Proposition 6 （b.a.u）, Proposition 5 (maximal)
%Yeadon

Let $(\m M,\tau)$ be a semi-finite von Neumann algebra. As usual, we
denote by $L_0(\m M,\tau)$ (or simply $L_0(\m M)$) the set of
$\tau$-measurable operators with respect to $\tau$ and by $L_p(\m M, \tau)$
(or simply $L_p(\m M)$)
the noncommutative $L_p$-spaces associated to $(\m M,\tau)$ for
$0<p\leq \infty$ ($L_\infty(\m M)=\m M$). We refer the reader to \cite{PiXu03, Xu} for measurable operators and noncommutative $L_p$-spaces.

We recall the notion of  almost uniform convergence (see e.g. \cite{Pad67}[Theorem 3.1], \cite{Lan76, Cuc71}) just for convenience.

\begin{defn}\label{def:au}
  Let  $(Y_n)_{n\geq 0}$ be a sequence in $ L_0(\m M)$. It converges almost uniformly if there exists some $Y\in L_0(\m M)$ such that for any $\varepsilon>0$ , there exists some projection $e\in \m M$ with
    $$\tau(1-e)<\varepsilon \qquad \qquad \textrm{and }\quad \quad \lim_{n\to \infty} \| (Y_n-Y) e\|_\infty =0.$$
     \end{defn}

We will use the associated notion of non-increasing rearrangement for a sequence (cf. \cite{CaRi}):
 
\begin{defn}
  Given a sequence $(Y_n)_{n\geq 0}$ in $L_0(\m M)$, we define the non-increasing function  $\mu_{\cdot}^c: \mathbb R^{+}\to [0,\infty]$ by, for $t>0$
  $$  \mu_t^c((Y_n)_{n\geq 0})= \inf_{\tau(1-e)\leq t}  \sup_{n\geq 0} \| Y_ne\|,$$
  where the infimum runs over all projections $e\in \m M$.
\end{defn}

This definition a priori depends on the ambient algebra $(\m M,\tau)$
and we should emphasize it by using $\mu_t^{\m M, c}$. Actually it can
be checked that if $(\m N,\tau')$ is another algebra containing $\m M$
with $\tau'_{|\m M}=\tau$ then
$$\mu_t^{\m N, c} ((Y_n)_{n\geq 0})\leq \mu_t^{\m M, c} ((Y_n)_{n\geq 0})
\leq C \mu_{\frac t 2}^{\m N, c} ((Y_n)_{n\geq 0}),$$ for some universal
$C>0$, any $t>0$ and any sequence $(Y_n)_{n\geq 0}$ in $L_0(\m M)$.
Thus the dependence on the algebra is rather mild and we drop it (we
will always stay in the same algebra anyway).

      We will freely use the usual inequality
      $$\mu_{t+s}^c ( (Y_n)_{n\geq 0}+(X_n)_{n\geq 0})\leq \mu_t^{c} ((Y_n)_{n\geq 0})+ \mu_s^{c} ((X_n)_{n\geq 0}).$$

The connection with  almost  uniform convergence is given by 
\begin{lem}\label{lem:finity au}
  Let $Y=(Y_n)_{n\geq 0}\subset L_p(\mathcal M)$ for some $0<p\leq \infty$. If
  $Y$ converges almost uniformly, then for any $t>0$, one has
  $\mu^c_t(Y)<\infty$.
\end{lem}

This can be verified easily by the definitions, we omit the details.

\section{Proof of Theorem \ref{thm:nonau}}

Let $N\in\mathbb N^*$. Consider the noncommutative probability space $M_N$, the algebra of matrices of dimension $N$, equipped with the normalized trace $\tau_N$, and the von Neumann subalgebras $M_n\oplus \ell_\infty^{{N-n}}$ for $1\leq n\leq N$ with the associated conditional expectation $\mathbb{E}_n$.

Consider a big algebra encoding all $(M_N, \tau_N)$,
$$(\mathcal M,\tau)=L_\infty(\{\pm1\}^\mathbb N)\overline{\bigotimes}\left(\bigotimes^\infty_{N=1}(M_N, \tau_N)\right),$$
and the associated von Neumann subalgebras for $n\geq 1$
$$\mathcal M_n=L_\infty(\{\pm1\}^\mathbb N)\overline{\bigotimes}\left(\bigotimes^n_{N=1}M_N\right)\overline{\bigotimes}\left(\bigotimes^\infty_{N=n+1}(M_n\oplus \ell_\infty^{{N-n}})\right)$$
with conditional expectations denoted by $\mathcal E_n$. We set
$\m M_0$ to be the copy of $L_\infty(\{\pm1\}^\mathbb N)$ with conditional expectation $\m E_0$.

In the $N^{\rm th}$ factor $(M_N, \tau_N)$, we denote by $e_{i,j}^{(N)}$ the usual matrix units. Let $X_N=\xi_N \xi^*_N\in M_N$ where $\xi_N=\sum^N_{k=1}e_{k,1}^{(N)}$. It is easy to check that $\|X_N\|_1=1$. To get the result in $L_p$ when $1\leq p<2$, we consider the normalization
$X_{p,N}=N^{\frac 1 p -1} X_N$, so that $\|X_{p,N}\|_p=1$.

The elements that we will be interested in are, for $1\leq p<2$
$$\mathcal X_p=\sum^\infty_{N=1}\varepsilon_N\frac{1}{N^2}X_{p,N!},$$
where $\varepsilon_n$ are the coordinates of $\{\pm1\}^\mathbb N$.
We will consider the martingale associated to  the filtration $(\mathcal M_n)^\infty_{n=0}$ and $\mathcal X_p\in L_p(\mathcal M)$.

\begin{prop}\label{prop:key estimate}
 There exist some  $t'>0$ and $\delta>0$  such that for all $0<t\leq t'$ and  $N\geq1 $,
\begin{align}\label{key estimate}
\mu^{c}_{t}( (\mathbb E_n(X_N))_{1\leq n\leq N})\geq \delta N^{\frac12},
\end{align}
and moreover for $1\leq p<2$ and $t\leq t'$, 
\begin{align}\label{equal8}
\mu^{c}_{\frac t 2}\big((\m E_n(\mathcal X_p))_{n\geq 0}\big)=\infty.
\end{align}
\end{prop}

We put off the proof of Proposition \ref{prop:key estimate} to next
section.  One may now conclude Theorem \ref{thm:nonau} directly from
Proposition \ref{prop:key estimate} thanks to Lemma \ref{lem:finity
  au}.

\section{Proof of the estimates}

\begin{proof}[Proof of Proposition \ref{prop:key estimate}]
Assuming \eqref{key estimate}, we first conclude \eqref{equal8}.

Define the trace preserving $*$-automorphism on $\mathcal M$ induced by
$\pi_N: \varepsilon_n\rightarrow (-1)^{\delta_{n=N}} \varepsilon_n$ on the $L_\infty(\{\pm1\}^\mathbb N)$ component, still denoted by $\pi_N$. Then one has the identity
\begin{align}\label{identity}\frac{2}{N^2}  X_{p,N!} = \varepsilon_{N}(\mathcal X_p-\pi_{N}(\mathcal X_p)).\end{align}
Therefore, we deduce that
\begin{align*}
\mu^c_t \big((\m E_n(\frac 2 {N^2} X_{p,N!}))_{n\geq 0}\big)&\leq \mu^c_t\big((\m E_n(\m X_p-\pi_{N}(\m X_p))_{n\geq 0}\big)\\
&\leq\mu^c_{\frac t 2}\big((\m E_n(\m X_p))_{n\geq 0}\big)+\mu^c_{\frac t 2}\big((\m E_n(\pi_{N}(\m X_p))_{n\geq 0}\big)=2\mu^c_{\frac t 2}\big((\m E_n(\m X_p))_{n\geq 0}\big),
\end{align*}
where we used that $\pi_{N}$ and $\mathcal E_n$ commute. Recall that
$X_{p,N!}=(N!)^{\frac 1 p-1}X_{N!}$, then applying  \eqref{key estimate} for  $t<t'$,
$$\mu^c_{\frac t 2}\big((\m E_n(\m X_p))_{n\geq 0}\big)\geq \delta (N!) ^{\frac 1p-\frac 12}N^{-2},$$
which yields \eqref{equal8} by letting $N\rightarrow\infty$.

\medskip

Now let us establish \eqref{key estimate}.  Fix $t\in ]0,1[$ and denote
$$\mu_t=\mu^c_{t}\big( (\mathbb E_n (X_N))_{n\geq 0}\big).$$ By the definition, there exists some projection $e\in\mathcal M$ such that for all $n\geq 0$
\begin{align}\label{defimu} \tau(1-e)\leq t\;\qquad \mathrm{and}\qquad  \|\mathbb E_n(X_N)e\|\leq 2\mu_t.\end{align}
Recall $X_N=\xi_N\xi^*_N$ with $\xi_N=\sum^N_{k=1}e_{k,1}^{(N)}$. By
easy computations, for $n\leq N$,
$$\mathbb E_n(X_N)= \sum^n_{k,l=1}e_{k,l}^{(N)}+\sum_{k=n+1}^N e_{k,k}:=Y_n+D_n.$$
Note that $Y_n=\eta_n\eta_n^*$ where $\eta_n=\sum^n_{k=1}e_{k,1}^{(N)}$. By multiplying by the projection $p_n=\sum_{k=1}^n e_{k,k}^{(N)}$ from the left, one gets for any $1\leq n\leq N$,
$$\|Y_ne\|\leq2\mu_t.$$
Now one may write, \begin{align*}
Y_n=Y_n(1-e)+Y_ne=Y_n(1-e)+2\mu_t U_n^*e=2\mu_t eU_n+(1-e)Y_n,
\end{align*}
where $U_n$ is some contraction in $\mathcal M$ and the last equality follows from the self-adjointness of $Y_n$. 

Then $Y_n\eta_n=n\eta_n$ yields
$$n \eta_n =2\mu_t e U_n\eta_n +n (1-e)\eta_n.$$
In $\m M \otimes M_N$, we set
\begin{align*}
A_N:=\sum^N_{n=1}n\eta_n \otimes e_{1,n}&=2\mu_t (e\otimes e_{1,1})\big(\sum^{N}_{n=1}U_n \eta_n \otimes e_{1,n}\big)\\
&\quad\quad+((1-e)\otimes e_{1,1})\big (\sum^N_{n=1} n\eta_n \otimes e_{1,n}\big)=:B_N+C_N.
\end{align*}

We need some estimates that we postpone to the next section. We use ${\rm tr}_N$ for the usual trace on $M_N$ with ${\rm tr}_N (1)=N$:
\begin{lem}\label{lem:lpnorm}
Let $0<q<\frac 12$. Then, for some constants $C_q,c_q>0$ there hold
\begin{enumerate}[{\rm (i)}]
\item $\quad\quad\quad\quad\quad\quad\quad\quad\quad\quad\displaystyle{\|A_N\|_{L_q((\mathcal M,\tau)\otimes (M_N,{\rm tr}_N))}\geq c_q N,}$
\item $\quad\quad\quad\quad\quad\quad\quad\quad\quad\quad\displaystyle \|B_N\|_{L_q((\mathcal M,\tau)\otimes (M_N,{\rm tr}_N))}\leq 2 \mu_t{N}^{\frac12},$
\item $\quad\quad\quad\quad\quad\quad\quad\quad\quad\quad\displaystyle \|C_N\|_{L_q((\mathcal M,\tau)\otimes (M_N,{\rm tr}_N))}\leq C_qt^{\frac 1{2q}}N.$
  \end{enumerate}
\end{lem}

Fix $q<\frac 12$, we may get thanks to the $q$-triangle inequality,
$\|A_N\|_q^q\leq \|B_N\|_q^q+\|C_N\|_q^q$. Thus, for some constants $d_q$, $e_q>0$,
\begin{align}\label{inter1}
N^q\leq d_q   \mu_t^q N^{\frac q2}+ e_q t^{\frac 12} N^{q}.
\end{align}
One may take $t'$ verifying $1 - e_q t'^{\frac 12} =\frac12$. For $t\leq t'$, we deduce from \eqref{inter1} that for some constant $f_q>0$, $\mu_t\geq f_q N^\frac12$.
\end{proof}

\section{Proof of Lemma  \ref{lem:lpnorm} }

Now let us prove Lemma \ref{lem:lpnorm} relying on

\begin{lem}\label{lem:lpnormtri}
  Let $T_n=\displaystyle{ \sum_{1\leq i\leq j\leq n} e_{i,j}\in M_n}$ for $n\geq 1$. Then for all $0<q< 1$,
$$   \Big(\frac {2n} {1-2^{q-1}}\Big)^{\frac 1 q} \geq \|T_n\|_{L_q(M_n, {\rm tr}_n)} \geq \Big(\frac n 2\Big)^{\frac 1 q}.   $$ 
\end{lem}

\begin{proof}
  The lower bound follows from the $q$-triangle inequality using the fact
  that $T_n-ST_n=Id_n$ where $S=\sum_{k=1}^{n-1} e_{k,k+1}$ and $\|S\|=1$.

  For the upper one, let $v_k=2^{-k}\|T_{2^k}\|_{L_q(M_{2^k}, {\rm tr}_{2^k})}^q$. One has $v_0=1$ and cutting $T_{2^{k+1}}$ into 4 pieces, $v_{k+1}\leq  v_k +
  2^{k(q-1)-1}$. Hence we get $v_k\leq \frac 1{1-2^{q-1}}$ for all $k\geq 0$ from which the bound follows comparing $n$ with a power of $2$.
\end{proof}

\begin{proof}
(i) We use basic inequalities and Lemma \ref{lem:lpnormtri},
 \begin{align*}
   \|A_N\|_{L_q((\mathcal M,\tau)\otimes (M_N,{\rm tr}_N))}&=\|\sum^N_{n=1}n\eta_n\otimes e_{1,n}\|_{L_q((M_N,\tau_N)\otimes (M_N,{\rm tr}_N))}\\
                                                     &=\frac{1}{N^{\frac1q}}\|\sum^N_{n=1}n\eta_n\otimes e_{1,n}\|_{L_q((M_N, {\rm tr}_N)\otimes (M_N,{\rm tr}_N))}\\
                                                     &\geq \frac{1}{N^{\frac1q}}\|\sum^N_{n={[\frac N 2]}}n\eta_n\otimes e_{1,n}\|_{L_q((M_N, {\rm tr}_N)\otimes (M_N,{\rm tr}_N))}\\
                                                     &\geq \frac{N}{2N^{\frac1q}}\|\sum^N_{n={[\frac N 2]}}\eta_n\otimes e_{1,n}\|_{L_q((M_N, {\rm tr}_N)\otimes (M_N,{\rm tr}_N))}\\
   \\                                                    &\geq \frac {N^{1-\frac 1 q}} 2 \|T_{[\frac N 2]}\|_{L_q((M_{[\frac N 2]}, {\rm tr}_{[\frac N 2]}))}\geq \frac{N}{2^{1+\frac 2 q}},
                 \end{align*}
 where the last line follows from the fact that  $T_{[\frac N 2]}$ is a sub-matrix of $\sum^N_{n={[\frac N 2]}}\eta_n\otimes e_{1,n}$ and the norm in Schatten classes do not depend on the size of a matrix if one enlarges it.

(ii) Let $r$ such that $\frac1q=\frac1r+\frac12$. Then
\begin{align*}
 \|B_N\|_{L_q((\mathcal M,\tau)\otimes (M_N,{\rm tr}_N))}&\leq 2 \mu_t \|e\otimes e_{1,1}\|_r\|\sum^{N}_{n=1}(U_n \eta_n \otimes e_{1,n})\|_2\\
 &\leq 2 \mu_t\left(\sum^{N}_{n=1}\|\eta_n\|^2_{L_2(\mathcal M,\tau)}\right)^{\frac12}\\
  &= 2 \mu_t\left(\sum^{N}_{n=1}\frac{n}{N}\right)^{\frac12}\leq 2 \mu_t{N}^{\frac12},
 \end{align*}
 where in the second inequality we have used $\|U_n\|_\infty\leq1$.

 (iii) Similar computations as in (i) give that for $q<1$
$$ \|A_N\|_{L_q((\mathcal M,\tau)\otimes (M_N,{\rm tr}_N))}\leq {N^{1-\frac 1 q}}  \|T_{N}\|_{L_q((M_{N}, {\rm tr}_{N}))}\leq N\Big(\frac{2}{1-2^{q-1}}\Big)^{\frac  1 q}.$$
Assuming $q<\frac 1 2$,
\begin{align*}
\|C_N\|_{L_q((\mathcal M,\tau)\otimes (M_N,{\rm tr}_N))}&\leq \|(1-e)\otimes e_{1,1}\|_{2q}\|A_N\|_{2q}\\
&\leq t^{\frac 1{2q}}{N}\Big(\frac{2}{1-2^{2q-1}}\Big)^{\frac 1{ 2 q}}.
\end{align*}
\end{proof}

\section{Ergodic averages \label{sec:erg}}

By an observation by Neveu, extended to the noncommutative setting by Dang-Ngoc \cite{Da} (see also \cite{JuXu07}), one can go from martingales to ergodic averages.

\begin{thm}\label{thm:nonau2}
 Let $1\leq p<2$. There exists a probability space $(\mathcal M,\tau)$, $\m X\in L_p(\m M,\tau)$
and a trace preserving unital completely positive map $T:(\mathcal M,\tau) \to(\mathcal M,\tau)$ such that $(M_n(T)(\mathcal X))_{n\geq 0}$ does not converge
  almost uniformly.
\end{thm}

\begin{proof}
  We use the construction in Theorem \ref{thm:nonau} where
  $\mathcal X=\m X^*\in (\m M ,\tau)$. We take
  $T=\sum_{n=0}^\infty (\alpha_{n+1}-\alpha_n) \mathcal E_n$, where
  $(\alpha_n)_{n\geq 0}$ is an increasing sequence with $\alpha_0=0$ and
  $\lim \alpha_n=1$. By Lemma 4.2 in \cite{JuXu07}, there exists a
  suitable choice of $(\alpha_n)$ so that for some increasing sequence of integers
  $(m_n)_{n\geq 0}$,
  $\sum_{n=0}^\infty \| M_{m_n}(T)(\m X)- \mathcal E_n(\m X)\|_p\leq
  1$.  Write $Z_n=M_{m_n}(T)(\m X)- \mathcal E_n(\m X)$ and let
  $t>0$. As $Z_n=Z_n^*$, if $e_n=1_{[0,t^{-\frac 1 p}]}(|Z_n|)$, we
  have $\|Z_ne_n\|\leq t^{-\frac 1 p}$ and
  $\tau(1-e_n)\leq t \|Z_n\|^p_p$ by the Markov inequality. Let
  $e=\wedge_{n\geq 0} e_n$, then
  $\tau(1-e)\leq \sum_{n=0}^\infty \tau(1-e_n)\leq t$ and
  $\|Z_ne\|\leq t^{-\frac 1 p}$ for all $n\geq 0$. Thus
  $\mu_t^c\big((Z_n)_{n\geq 0}\big)\leq t^{-\frac 1 p}$ for all
  $t>0$. Since
  $\mu_t^c\big((\mathcal E_n(\m X))_{n\geq 0}\big)=\infty$ for $t$
  small, the same must hold for
  $\mu_t^c\big((M_{m_n}(T)(\m X))_{n\geq 0}\big)\leq \mu^c_t\big((M_{n}(T)(\m
  X))_{n\geq 0}\big)$. We can conclude that $(M_{n}(T)(\m X))_{n\geq 0}$
  does not converge a.u. as before thanks to Lemma \ref{lem:finity au}. 
\end{proof}

Actually, one can also produce $T$ as a trace preserving
$*$-automorphism and thus get a counter-example for the a.u.
convergence in the $L_p$-version of the original Lance theorem for
$1\leq p<2$.  We keep the notation of the previous sections. On $M_N$,
consider the diagonal unitary
$U_N=\sum_{k=1}^N e^{2\pi i K^{k-N-1}}e_{k,k}^{(N)}$, where
$K\in \mathbb N\setminus\{0,1\}$,  will be chosen big enough depending on $N$ later. We take
$T_N$ the conjugation by $U_N$. Let us evaluate the ergodic average
$y=M_{K^n}(T_N)x=\frac 1 {K^n}\sum_{k=0}^{K^n-1} U_N^kxU_N^{*k}$ with
$x=(x_{\alpha,\beta})_{1\leq \alpha,\beta\leq N}$. Clearly,
$y_{\alpha,\alpha}=x_{\alpha,\alpha}$ and if $\alpha\neq \beta$:
$$y_{\alpha,\beta} = \frac 1{K^n}\frac {1-e^{2\pi i (K^{\alpha-1-N}-K^{\beta-1-N})K^n}}
{1-e^{2\pi i (K^{\alpha-1-N}-K^{\beta-1-N})}}
x_{\alpha,\beta}=c_{\alpha,\beta}x_{\alpha,\beta}.$$ Thus if
$\alpha,\beta> N-n$ then clearly $y_{\alpha,\beta}=0$.

We deal with $\alpha\leq N-n<\beta$. There are constant $c_1, c_2>0$ such
that for all $z\in \mathbb C$ with $|z|\leq 4\pi$,
$\frac 1 {c_1}|z|\leq |e^{z}-1|\leq c_2 |z| $, thus we get
$$|c_{\alpha,\beta}|=\frac 1 {K^n}\Big|\frac {1-e^{2\pi i K^{\alpha-1-N+n}}}{1-e^{2\pi i (K^{\alpha-1-N}-K^{\beta-1-N})}}\Big| \leq 2c_1c_2\, K^{\alpha-\beta}\leq 2 c_1c_2 K^{-1}.$$ Similarly for $\beta\leq N-n<\alpha$, one also has $|c_{\alpha,\beta}|\leq 2c_1c_2 K^{-1}$.

When $\alpha,\beta\leq N-n$, we can use that there are constants 
$c_3, c_4>0$ such that for all $z\in \mathbb C$ with $|z|\leq 4\pi$, $\frac 1{c_3}|z|^2\leq |e^{z}-1-z|\leq c_4 |z|^2 $ and basic computations to get that for some $c_5>0$, $  |c_{\alpha,\beta}-1|\leq c_5 K^{-1}$.

For any $x\in L_p(M_N,\tau_N)$, with $C=\max\{2c_1c_2,c_5\}$, by the triangle inequality
$$\|\big(M_{K^n}(T_N)-\mathbb E_{N-n}\big) x\|\leq 
\sum_{\alpha,\beta=1}^{N-n} |c_{\alpha,\beta}-1| |x_{\alpha,\beta}| +
\sum_{\substack{N\geq \max\{\alpha,\beta\}>N-n\\\alpha\neq \beta}} |c_{\alpha,\beta}|
|x_{\alpha,\beta}|\leq C N^2K^{-1}\|x\|_p.$$ Thus by choosing $K$
big enough, we can assume that for $x\in L_p(M_N,\tau_N)$:
$$\sum_{n=0}^N \|\mathbb E_{N-n}(x)-  \frac 1 {K^n}\sum_{k=0}^{K^n-1} U_N^kxU_N^{*k}\|_{p}\leq  \|x\|_p,\qquad   \|U_N-1\|\leq 2^{-N}.$$
Then the product $\prod_{N=1}^\infty U_N$ converges for the norm in $\m
M$ to a unitary $U$. 
The $*$-automorphism $T$ on $(\m M,\tau)$ implemented by $U$
satisfies that $M_n(T)(x)=\frac 1 {n}\sum_{k=0}^{n-1} U_N^kxU_N^{*k}$ for $x\in M_N$. Then, by invoking identity \eqref{identity} and using the similar argument as in the proof of Theorem \ref{thm:nonau2}, the element
$\m X_p\in L_p(\m M)$ satisfies that
$$N^2\mu^c_{\frac t4}\big((M_{n}(T)(\m  X_p))_{n\geq 0}\big)\geq 
\mu^c_{\frac t 2}\big((M_{n}(T)(  X_{p,N!}))_{n\geq 0}\big)\geq
\mu^c_{t}\big((\mathbb E_n(  X_{p,N!}))_{N!\geq n\geq 0}\big)-2^{\frac 1 p}t^{-\frac 1 p}.$$
Finally one concludes with Proposition \ref{prop:key estimate}.

\subsection*{Acknowledgments} The first author was supported by
National Natural Science Foundation of China (No. 12071355,
No. 12325105, No. 12031004) and the second author  by
ANR-19-CE40-0002.

\end{document}